\documentclass{amsart}

\newtheorem{theorem}{Theorem}[section]
\newtheorem*{theorem*}{Theorem}

\newtheorem{proposition}[theorem]{Proposition}
\newtheorem{conjecture}[theorem]{Conjecture}

\theoremstyle{definition}
\newtheorem{definition}[theorem]{Definition}
\newtheorem*{defin*}{Definition}
\newtheorem*{example*}{Example}
\newtheorem{example}[theorem]{Example}

\theoremstyle{remark}
\newtheorem*{remark*}{Remark}
\newtheorem{remark}[theorem]{Remark}

\numberwithin{equation}{section}

\begin{document}
\title[Symplectic analog of Calabi's conj. for Calabi--Yau
$3$-folds] {Symplectic analog of Calabi's conjecture for Calabi--Yau
threefolds}
\author{Dmitry V. Egorov}
\address{
Ammosov Northeastern federal university\\ Kulakovskogo str. 48, 677000, Yakutsk, Russia}%
\email{egorov.dima@gmail.com}%
\thanks{This work is supported in part by the Council of the Russian
Federation President Grants (projects NSh-544.2012.1 and
MK-842.2011.1).}



\date{}


\begin{abstract}

In this paper we state an analog of  Calabi's conjecture proved by
Yau. The difference with the classical case is that  we propose
deformation of the complex structure, whereas the complex
Monge--Amp\`{e}re equation describes deformation of the K\"{a}hler
(symplectic) structure.
\end{abstract}

\maketitle

\section{Conjecture}

Celebrated  Calabi's  conjecture \cite{Calabi} proved by Yau
\cite{Yau} gives a criterion for a K\"{a}hler manifold  to have a
vanishing Ricci curvature. It states that if the first Chern class
of a manifold vanishes, then in a given K\"{a}hler class there
exists a unique Ricci-flat metric. Yau proved conjecture by showing
an existence of a solution for the complex Monge--Amp\`{e}re
equation.

In this paper we state an analog of Calabi's conjecture for
threefolds. We consider the deformation of a holomorphic volume
form, whereas  the K\"{a}hler form is not deformed. We clarify our
result of \cite{egorov3}, where the main equation is non-scalar.

\begin{conjecture}\label{conj:1}
Let $M$ be a compact K\"{a}hler manifold of complex dimension $n=3$
with $c_1(M)=0$. By $\omega$ and $\Omega=\rho+\sqrt{-1}\sigma$
denote K\"{a}hler form and holomorphic volume form on $M$
respectively. Consider equation
\begin{equation}\label{eq:main}
(\Omega - \sqrt{-1}dd^s\varphi\Omega)\wedge (\bar{\Omega} +
\sqrt{-1}dd^s\varphi\bar{\Omega}) =
e^F\Omega\wedge\bar{\Omega}\tag{$\ast$},
\end{equation}
or equivalently
$$
(\rho+dd^s\varphi\sigma)\wedge(\sigma -dd^s\varphi\rho) =
e^F\rho\wedge\sigma,
$$
where function $F:M\to\mathbb{R}$ is normalized so
that $\int_M{e^F}\omega^n = \int_M{\omega^n}$.

Then equation \eqref{eq:main} has unique (modulo adding a constant)
solution $\varphi:M\to\mathbb{R}$ such that
\begin{enumerate}
\item $\tilde{\rho} = \rho+dd^s\varphi\sigma$  is positive; \item $\tilde{\sigma} = \sigma -
dd^s\varphi\rho$ is negative;
\item $\tilde{\rho}$ is dual to $\tilde{\sigma}$.
\end{enumerate}
\end{conjecture}

We state  notions of positivity and duality for $3$-forms and
definition for $d^s$ operator later in the paper. Now let us restate
Conjecture \ref{conj:1} in a form free of PDE.

\begin{conjecture}\label{conj:2}
Let $M$ be a compact K\"{a}hler manifold of complex dimension $n=3$
with $c_1(M)=0$; let $\omega$ and $\Omega$ be a K\"{a}hler and
holomorphic volume forms respectively. Then there exists a unique
$\tilde{\Omega}\sim\Omega$ such that the K\"{a}hler metric $g_{ij} =
\omega_{ik}\tilde{J}^k_j$ is Ricci-flat. Here $\tilde{J}$ is the
complex structure determined by $\tilde{\Omega}$.
\end{conjecture}

The equivalence of Conjectures \ref{conj:1} and \ref{conj:2} follows
from the following definitions of stable forms. The notion of a
stable $3$-form was introduced by Hitchin in \cite{Hitchin3f,
Hitchin_stable}.

\begin{definition}A real $3$-form $\rho$ on real  $6$-space $V$ is called
{\it stable} if the orbit of $\rho$ under the natural action of
$GL(V)$ is open in $\Lambda^3V^*$.
\end{definition}

Any stable form canonically determines a complex structure on $V$.

\begin{definition}A real $3$-form $\rho$ on real symplectic $6$-space is called
{\it positive} if
\begin{enumerate}
\item $\rho$ is stable;
\item $\rho$ is primitive;
\item $\omega(X, J_\rho X)\geq 0$ for any $X\in V$, where $\omega$ is a
symplectic form and $J_\rho$ is the complex structure determined by
$\rho$.
\end{enumerate}

\end{definition}

Any positive (negative) closed $3$-form is a real (imaginary) part
of the holomorphic volume form.

\begin{definition}A stable $3$-form $\rho$ is called {\it dual} to a stable
$3$-form $\sigma$ if
$$
\sigma = J_\rho\rho.
$$
\end{definition}


Since we do not deform the K\"{a}hler structure,   a symplectic
differential operator is used.

\begin{definition}
\begin{equation*}
d^s\alpha := (-1)^{k+1}\ast_s d\,\ast_s\alpha,
\end{equation*}
where $\alpha$ is a $k$-form and $\ast_s$ is the symplectic Hodge
star. The definition of $\ast_s$ is obtained by substituting the
Riemannian metric for symplectic form in the definition of the usual
Hodge star.

\end{definition}
\begin{remark}
Note that $d^s$ decreases degree of the form by one and anticommutes
with $d$. There is also the $dd^s$-lemma analogous to the
$dd^c$-lemma. Both lemmata hold on a K\"{a}hler manifold. See for
example \cite{Tseng} and references therein.
\end{remark}

\begin{proposition}
The equation \eqref{eq:main} differs from the complex
Monge--Amp\`{e}re equation.
\end{proposition}

\begin{proof}

Suppose $U$ is an open set of $M$ with coordinate chart $\{x^i\}$
such that $\omega = dx^{12}+dx^{34}+dx^{56}$ on $U$; then the global
equation \eqref{eq:main} is locally equivalent to the following
equation on $U$
\begin{equation*}
\begin{array}{l}
\quad(\varphi_{22}+\varphi_{33}+\varphi_{55})(\varphi_{11}+\varphi_{44}+\varphi_{66})
+\ (\varphi_{11}+\varphi_{44}+\varphi_{55})(\varphi_{22}+\varphi_{33}+\varphi_{66})\\
+\
(\varphi_{11}+\varphi_{33}+\varphi_{66})(\varphi_{22}+\varphi_{44}+\varphi_{55})
+\
(\varphi_{22}+\varphi_{44}+\varphi_{66})(\varphi_{11}+\varphi_{33}+\varphi_{55})\\
-\ (\varphi_{12}+\varphi_{34}+\varphi_{56})^2 -\
(-\varphi_{12}-\varphi_{34}+\varphi_{56})^2
-(\varphi_{12}-\varphi_{34}-\varphi_{56})^2\\
-\ (-\varphi_{12}+\varphi_{34}-\varphi_{56})^2
-2\left[(\varphi_{13}-\varphi_{24})^2
+(\varphi_{36}+\varphi_{45})^2+(\varphi_{15}-\varphi_{26})^2\right.\\
+\ \left.(\varphi_{16}+\varphi_{25})^2+(\varphi_{35}-\varphi_{46})^2
+(\varphi_{14}+\varphi_{23})^2\right] = 1,
\end{array}\end{equation*}
where  $\varphi_{ij} =
\partial^2\varphi/\partial x^i\partial x^j$.

Obviously, this equation  differs from the complex Monge--Amp\`{e}re
equation. In fact, the operator of the equation coincides with   a
level-$p$ Monge--Amp\`{e}re operator considered in \cite{HL2}.
\end{proof}

\begin{remark}
The uniqueness of the $n=3$ case is that a single stable real
$3$-form determines complex structure
\cite{Hitchin3f,Hitchin_stable}. In higher dimensions one needs a
pair of real $n$-forms to determine a complex structure and
therefore a pair of functions.
\end{remark}

\begin{remark}In  \cite{Hull} the generalized Monge--Amp\`{e}re
equation is defined in the sense of generalized complex structures.
However, for the case of the K\"{a}hler geometry it reduces to the
usual complex Monge--Amp\`{e}re equation.
\end{remark}

\section{Mirror K\"{a}hler potential}

In paper \cite{egorov3} we proposed a new equation for the
$3$-dimensional Calabi--Yau metrics and proved the  solution
existence theorem.

\begin{theorem}\label{th:old}
Let $(M,\omega)$ be a compact K\"{a}hler $3$-manifold such that
$c_1(M)=0$; let $\Omega = \rho+i\sigma$ be a holomorphic volume form
on $M$. Then  the following equation
\begin{equation*}
(\rho+dd^s\alpha)\wedge(\sigma +dd^s\beta) = e^F\rho\wedge\sigma
\end{equation*}
or equivalently
\begin{equation*}(\Omega + dd^s\psi)\wedge (\bar{\Omega} + dd^s\bar{\psi}) =
e^F\Omega\wedge\bar{\Omega},\quad \psi = \alpha+\sqrt{-1}\beta
\end{equation*}
has solution $\alpha, \beta\in \Omega^3(M)$ such that
\begin{enumerate}
\item $\rho+dd^s\alpha$ and $\sigma +dd^s\beta$ are primitive stable $3$-forms;
\item $\rho+dd^s\alpha$ is dual to $\sigma +dd^s\beta$  in
the sense of the stable forms;
\item real function $F$ is normalized: $\int_M{e^F\rho\wedge\sigma} = \int_M
\rho\wedge\sigma$.
\end{enumerate}
\end{theorem}

However, the main equation of \cite{egorov3} is non-scalar. We
relate Theorem \ref{th:old} with Conjecture \ref{conj:1} by
conjecturing that there exists a {\it mirror K\"{a}hler potential}.

\begin{conjecture}\label{conj:3}Let $M$ be a
compact Calabi--Yau $3$-fold with K\"{a}hler form $\omega$ and
holomorphic volume form $\Omega=\rho+\sqrt{-1}\sigma$. Then for any
$\tilde{\Omega}\sim\Omega$ there exists a function
$\varphi:M\to\mathbb{R}$ such that
$$
\tilde{\Omega} = \Omega + \sqrt{-1}dd^s\varphi\Omega.
$$
\end{conjecture}

\begin{example}Suppose $\mathbb{C}^3$ with a flat Hermitian
metric $\sum dz^i\otimes d\bar{z}^i$; the K\"{a}hler form $\omega =
(\sqrt{-1}/2)\sum dz^i\wedge d\bar{z}^i$ and holomorphic volume form
$\Omega = dz^1\wedge dz^2\wedge dz^3$. Then the following identities
hold:
$$
\omega = dd^c\varphi, \quad \Omega = -
(\sqrt{-1}/3)dd^s\varphi\:\Omega,
$$
where $\varphi = \sum|z^i|^2$.
\end{example}

%

It would be interesting to establish a relationship between the
classical and mirror K\"{a}hler potentials.

The author would like to thank  D.~Alexeevsky, G.~Cavalcanti and
M.~Verbitsky for useful remarks during the conference "Geometric
structures on complex manifolds -- 2011" in Moscow. The author is
grateful to I.\,A.~Taimanov for support and useful remarks.

\bibliographystyle{amsplain}

\begin{thebibliography}{99}






\bibitem{Calabi} E. Calabi,
{\it On K\"{a}hler manifolds with vanishing canonical class},
Algebr. Geom. Topol. A symposium in honor of S. Lefshetz, pp.
78--89, Princeton University Press, Princeton, N.J., 1957.

\bibitem{Yau}S.-T. Yau, {\it On the Ricci curvature of compact K\"{a}hler manifold and the
complex Monge--Amp\`{e}re equation I},  Commun. pure appl. math.
{\bf 31} (1978), 339--411.


\bibitem{egorov3}D. Egorov, {\it New equation on the low-dimensional
Calabi--Yau metrics}, Sib. Math. J. {\bf 52} (2011), no. 4, 651--654

\bibitem{Hitchin3f} N.J.Hitchin,
{\it The geometry of three-forms in six dimensions}, J. Diff. Geom.
{\bf 55} (2000), 547--576.


\bibitem{Hitchin_stable} N.J.Hitchin, {\it Stable forms and special metrics},
"Global Differential Geometry: The Mathematical Legacy of Alfred
Gray", M. Fern\'{a}ndez and J. A. Wolf (eds.), Contemp. Math. {\bf
288}, AMS, Providence (2001).


\bibitem{Tseng}Li-Sheng Tseng and S.-T. Yau, {\it Cohomology and Hodge theory on
symplectic manifolds: I}, arXiv:0909.5418 (2009).


\bibitem{HL2}   F. Reese Harvey and H. Blaine Lawson Jr.
{\it Foundations of $p$-convexity and $p$-plurisubharmonicity in
Riemannian geometry}, arXiv:1111.3895 (2011).


\bibitem{Hull}Chris M. Hull, Ulf Lindstr\"{o}m, Martin Ro\v{c}ek, Rikard von
Unge and Maxim Zabzine, {\it Generalized Calabi--Yau metric and
generalized Monge--Ampere equation}, J. High Energy Phys. {\bf 2010}
(2010), no. 8, 1--23.








%
%
%
%
%





\end{thebibliography}

\end{document}